\documentclass[11 pt]{amsart}

\usepackage{amsaddr}
\usepackage{amssymb}
\usepackage{asymptote}
\usepackage{float}
\usepackage[margin=1.5in]{geometry}
\usepackage{graphicx}
\usepackage{setspace}
\usepackage[parfill]{parskip}
\usepackage{todonotes}

\makeatletter
\newtheorem*{rep@theorem}{\rep@title}
\newcommand{\newreptheorem}[2]{%
\newenvironment{rep#1}[1]{%
 \def\rep@title{#2 \ref{##1}}%
 \begin{rep@theorem}}%
 {\end{rep@theorem}}}
\makeatother

\theoremstyle{definition}
\newtheorem{defn}{Definition}[section]
\newtheorem{ex}{Example}[section]
\newtheorem{alg}{Algorithm}[section]
\newtheorem{prob}{Problem}[section]

\theoremstyle{plain}
\newtheorem{thm}{Theorem}[section]
\newtheorem{conj}[thm]{Conjecture}
\newtheorem{lem}[thm]{Lemma}

\newtheorem{prop}[thm]{Proposition}
\newreptheorem{thm}{Theorem}

\theoremstyle{remark}
\newtheorem{rem}{Remark}[section]

\newcommand{\vocab}[1]{\textit{#1}}

\numberwithin{equation}{section}

\newcommand{\f}{\frac}
\newcommand{\lt}{\left}
\newcommand{\rt}{\right}

\newcommand{\fS}{{\mathfrak{S}}}

\DeclareMathOperator{\rred}{Rred}
\DeclareMathOperator{\red}{Red}
\DeclareMathOperator{\eq}{Eq}

\begin{document}

\title[$(132,213)$-avoiding cyclic permutations]{An upper bound on the number of $(132,213)$-avoiding cyclic permutations}

\author{Brice Huang}
\address{Department of Mathematics, MIT, Cambridge, MA, USA.}
\email{bmhuang@mit.edu}


\date{August 25, 2018}

\begin{abstract}
    We show a $n^2 \cdot 2^{n/2}$ upper bound
    on the number of $(132,213)$ avoiding cyclic permutations.
    This is the first nontrivial upper bound on the number of such permutations.
    We also construct an algorithm to determine
    whether a $(132,213)$ avoiding permutation is cyclic
    that references only the permutation's layer lengths.
\end{abstract}

\maketitle

\section{Introduction}\label{sec:intro}

The theory of \vocab{pattern avoidance} in permutations
has been widely studied since its introduction by Knuth in 1968 \cite{Knu}.
The classical form of this problem asks to count the number of permutations
of $[n] = \{1,\ldots,n\}$ avoiding a given pattern $\sigma$.
Since then, many variations of this problem have been proposed and studied
in the literature.

We will focus on the problem of pattern avoidance among permutations
consisting of \vocab{a single cycle}.
This problem was first posed by Stanley in 2007
at the Permutation Patterns Conference,
and was subsequently studied by
Archer and Elizalde \cite{AE} and B\'ona and Cory \cite{BC}.

Let us first recall the definition of pattern avoidance.
Let $\fS_n$ denote the set of permutations of $[n]$.
\begin{defn}
    Let $\sigma \in \fS_k$.
    A permutation $\pi\in \fS_n$ \vocab{contains} $\sigma$ if
    there exist indices $1\le i_1< \ldots < i_k\le n$ such that the sequence
    \[
        \pi_{i_1}\pi_{i_2} \cdots \pi_{i_k}
    \]
    is in the same relative order as $\sigma_1\sigma_2\cdots \sigma_k$.
    Otherwise, $\pi$ \vocab{avoids} $\sigma$.
\end{defn}
In either case, $\sigma$ is the \vocab{pattern} that $\pi$ contains or avoids.
\begin{ex}
    The permutation $12534$ contains the pattern $132$
    and avoids the pattern $321$.
\end{ex}

In classical pattern avoidance,
the central objects of study are the numbers $A_n(\sigma)$,
defined as follows.
\begin{defn}
    Let $\sigma\in \fS_k$ be a pattern.
    Then, $A_n(\sigma)$ is the number of permutations $\pi \in \fS_n$
    avoiding the pattern $\sigma$.
    Similarly, for $(\sigma_1,\sigma_2)\in \fS_{k_1}\times \fS_{k_2} $, $A_n(\sigma_1,\sigma_2)$
    is the number of permutations $\pi \in \fS_n$ avoiding both $\sigma_1$ and $\sigma_2$.
\end{defn}

In a classical result, Knuth \cite{Knu} showed that
$A_n(\sigma) = \f{1}{n+1} \binom{2n}{n}$, the $n^{\text{th}}$ Catalan number,
for any pattern $\sigma \in \fS_3$.
In 1985, Simion and Schmidt \cite{SiSc} proved analogous results for
permutations avoiding two patterns, computing the value of $A_n(\sigma_1,\sigma_2)$
for any pair of distinct patterns $(\sigma_1,\sigma_2)\in \fS_3\times \fS_3$.
For an overview of related results in classical pattern avoidance,
see the book by Linton, Ru\v{s}cuk, and Vatter \cite{LRV}.

In classical pattern avoidance, we think of permutations only as linear orders.
We can also think of permutations algebraically,
in terms of their cycle decompositions.
For example, the permutation whose one-line notation is $24513$
has cycle decomposition $(124)(35)$.
With this perspective, we can discuss pattern avoidance among permutations
whose cycle decompositions consist of a single cycle.

\begin{defn}
    Let $\sigma \in \fS_k$ be a pattern.
    Then, $C_n(\sigma)$ is the number of permutations $\pi\in \fS_n$
    avoiding $\sigma$ that consist of a single $n$-cycle.
    Similarly, for $(\sigma_1,\sigma_2)\in \fS_{k_1}\times \fS_{k_2}$, $C_n(\sigma_1,\sigma_2)$
    is the number of permutations $\pi \in \fS_n$ avoiding both $\sigma_1$ and $\sigma_2$
    that consist of a single $n$-cycle.
\end{defn}
In 2007, Richard Stanley asked for the determination
of the value of $C_n(\sigma)$ for any $\sigma\in \fS_3$.
All cases of this problem remain open;
this problem is difficult because it requires considering
both views of permutations described above.

Work on the analogous problem for cyclic permutations avoiding two patterns
was begun by Archer and Elizalde in 2014.
They showed the following result.
\begin{thm}\cite{AE}
    Let $\mu$ be the number-theoretic M\"obius function.  For all positive integers $n$,
    \[
        C_n(132,321) =
        \f{1}{2n}\sum_{\substack{d|n \\ \text{$d$ odd}}} \mu(d) 2^{n/d}.
    \]
\end{thm}
This result was proved by studying permutations realized by shifts.
For more results about such permutations, see \cite{AEK, Eli1, Eli2}.
B\'ona and Cory \cite{BC} determined $C_n(\sigma_1,\sigma_2)$
for several pairs $(\sigma_1,\sigma_2)\in \fS_3\times \fS_3$.
Their most significant result is as follows.
\begin{thm}\cite{BC}
    Let $\phi$ be the Euler totient function.  For all positive integers $n$,
    \[
        C_n(123,231) =
        \begin{cases}
            \phi(\f n2) & n\equiv 0\pmod 4, \\
            \phi(\f{n+2}{4}) + \phi(\f n2) & n\equiv 2 \pmod 4, \\
            \phi(\f{n+1}{2}) & n\equiv 1\pmod 2.
        \end{cases}
    \]
\end{thm}
B\'ona and Cory also proved the following formulae for other pairs $(\sigma_1,\sigma_2) \in \fS_3\times \fS_3$.
\begin{thm}\cite{BC}
    The following identities hold.
    \begin{itemize}
        \item For all $n\ge 5$, $C_n(123,321)=0$.
        \item For all $n\ge 3$, $C_n(231,312)=0$.
        \item For all positive integers $n$, $C_n(231,321)=1$.
        \item For all positive integers $n$, $C_n(132,321)=\phi(n)$.
        \item For all positive integers $n$, $C_n(123,132)=2^{\lfloor (n-1)/2\rfloor}$.
    \end{itemize}
\end{thm}
Other results regarding pattern avoiding permutations with given cyclic structure can be found in
\cite{Gan, GeRe, Thi, WeRo}.

The above formulae by \cite{AE, BC} and analogous formulae obtained from them by symmetry
determine the values of $C_n(\sigma_1,\sigma_2)$
for each pair of distinct patterns $(\sigma_1,\sigma_2)\in \fS_3\times \fS_3$ except $(132,213)$.
This motivates the following problem.
\begin{prob}\label{prob:main}
    Determine an explicit formula for $C_n(132,213)$.
\end{prob}
Solving this problem would complete the enumeration of
cyclic permutations avoiding pairs of patterns of length 3.

The main result of this paper is the following bound.
\begin{thm}\label{thm:main}
    For all positive integers $n$, $C_n(132,213) \le n^2 \cdot 2^{n/2}$.
\end{thm}
To our knowledge,
this is the first nontrivial upper bound of $C_n(132,213)$,
though computer experiments suggest this bound is not asymptotically tight.

A \vocab{composition} of $n$ is a tuple $(a_1,\ldots,a_k)$ of positive integers
with sum $n$.
The $(132,213)$ avoiding permutations of size $n$
are the so-called \vocab{reverse layered permutations},
which correspond to the compositions of $n$.
We refer to compositions corresponding to cyclic permutations as cyclic compositions.
The second main result of this paper is the following algorithm,
which determines, without reference to the associated permutation,
whether a composition $C=(a_1,\ldots,a_k)$ is cyclic.
\begin{alg}\label{alg:main}
    Take as input a composition $C$.
    Run the \vocab{repeated reduction algorithm} $\rred$
    on the \vocab{equalization} $\eq (C)$.
    If $\rred$ outputs that $\eq(C)$ is cyclic, output that $C$ is cyclic.
    Otherwise, output that $C$ is not cyclic.
\end{alg}
The operations $\rred$ and $\eq$ are defined in Sections~\ref{sec:balanced}
and \ref{sec:unbalanced}, respectively.
This result is interesting in its own right,
and is the key step in the proof of Theorem~\ref{thm:main};
it implies that any cyclic composition has one or two odd terms,
which gives the bound in Theorem~\ref{thm:main}.

The rest of this paper is structured as follows.
In Section~\ref{sec:prelim} we formalize the connection between
$(132,213)$-avoiding permutations and compositions of $n$.
We introduce the notions of
\vocab{balanced compositions} and \vocab{cycle diagrams},
tools that will be useful in our analysis.
In Section~\ref{sec:balanced} we prove our results for balanced permutations.
In Section~\ref{sec:unbalanced} we generalize our results
to all permutations.
Finally, in Section~\ref{sec:conclusion}
we present some directions for further research.

\section{Preliminaries}\label{sec:prelim}

\subsection{Reverse Layered Permutations}

The \vocab{skew sum} $\pi \ominus \tau$ of permutations
$\pi\in \fS_m$, $\tau \in \fS_n$ is defined by
\[
(\pi\ominus \tau)(i) =
\begin{cases}
\pi(i)+n & i \le m \\
\tau(i-m) & i>m
\end{cases}
\]
for all $i\in [m+n]$.
Note that $\ominus$ is an associative operation.
Moreover, let $I_n=123\cdots n$ denote the identity permutation on $n$ elements.

A \vocab{reverse layered permutation} is a permutation of the form
\[\pi = I_{a_1}\ominus I_{a_2}\ominus \cdots \ominus I_{a_k},\]
for some positive integers $a_1,\ldots,a_k$ summing to $n$.
Explicitly, a reverse layered permutation is of the form
\[\pi = (n-a_1+1)\cdots (n) ~ (n-a_1-a_2+1)\cdots (n-a_1) ~ \cdots ~ (1)(2)\cdots (a_k).\]
It is known that the $(132,213)$ avoiding permutations
are the reverse layered permutations \cite{Bon}.
We can bijectively identify these permutations with
the compositions $(a_1,\ldots,a_k)$ of $n$.

As an immediate consequence,
there are $2^{n-1}$ reverse layered permutations of length $n$,
corresponding to the $2^{n-1}$ compositions of $n$.
It remains, therefore, to determine the number of these permutations
that are also cyclic.

Recall that a composition of $n$ is \vocab{cyclic}
if its associated reverse layered permutation is cyclic.
Determining $C_n(132,213)$ is therefore equivalent
to counting the cyclic compositions of $n$.

\subsection{Balanced Compositions}

We say a composition $C=(a_1,\ldots,a_k)$ of $n$ is \vocab{balanced} if
some prefix $a_1,\ldots,a_j$ has sum $\f n2$.
We denote such compositions with the notation
$(a_1,\ldots,a_j | a_{j+1},\ldots , a_k)$.
Otherwise, we say $C$ is \vocab{unbalanced}.
Note that compositions of odd $n$ are all unbalanced.

Equivalently, $C$ is balanced if its
associated reverse layered permutation $\pi$
has the property that $\pi(i) > \f n2$ if and only if $i\le \f n2$.
We say a reverse layered permutation is \vocab{balanced}
if it has this property, and \vocab{unbalanced} otherwise.

\begin{ex}
    The composition $C=(1,2,1,2)$ is balanced, and corresponds to the
    reverse layered permutation $\pi=645312$.
    The permutation $\pi$ has the property that $\pi(i)>3$ if and only if $i\le 3$.
    The composition $C'=(1,3,2)$ is unbalanced, and corresponds to the
    reverse layered permutation $\pi'=634512$.
    The permutation $\pi'$ does not have this property.
\end{ex}

\subsection{Cycle Diagrams}

Define the \vocab{graph} of a permutation $\pi\in \fS_n$
as the collection of points $(i, \pi(i))$, for $i\in [n]$.
The \vocab{cycle diagram} of $\pi$ is obtained from the graph of $\pi$
by drawing vertical and horizontal line segments, called \vocab{wires},
from each point in the graph to the line $y=x$.
We say the points $(i, \pi(i))$ are the \vocab{points} of the cycle diagram.
By slight abuse of notation, we say the cycle diagram of a composition $C$
is the cycle diagram of its associated permutation.

The wires in the cycle diagram of a permutation $\pi$ form one or more contiguous loops.
Each loop has the property that, when followed clockwise,
the column it visits after the $i^{\text{th}}$ column is the $\pi(i)^{\text{th}}$ column.
Thus, the loops of the cycle diagram of $\pi$ coincide with the cycle decomposition of $\pi$;
in particular, a permutation is cyclic if and only if
the wires in its cycle diagram make a single closed loop.
Moreover, a permutation is balanced if and only if each point in its cycle diagram is,
along the wire path, adjacent to two points on the opposite side of the line $y=x$.

Each layer $I_{a_i}$ in a reverse layered permutation corresponds
to a layer of $a_i$ diagonally-adjacent points in the cycle diagram
running from bottom left to top right;
successive layers are ordered from top left to bottom right.

\begin{ex}
    Figure~\ref{fig:cycle-diagram} shows the cycle diagrams of
    the balanced cyclic permutation $645312$,
    the unbalanced cyclic permutation $53412$,
    and the balanced noncyclic permutation $456321$.
    In the cycle diagram of $645312$, the wire forms a single loop, and
    each point (marked with $\times$) is, along the wire path,
    adjacent to two points on the opposite side of the line $y=x$.
    In the cycle diagram of $53412$, the wire still forms a single loop, but
    the two points in the middle layer are adjacent to each other on the wire path,
    and not to points on the opposite side of the line $y=x$.
    In the cycle diagram of $456321$, the wire does not form a single loop.
    \begin{figure}
        \centering
        \includegraphics{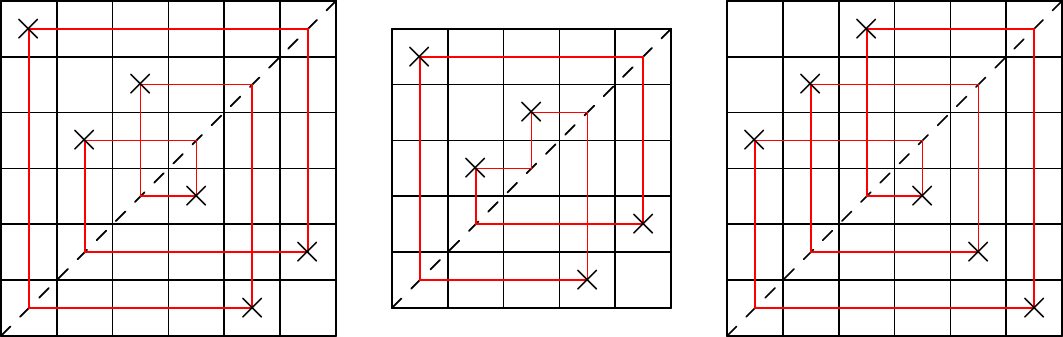}
        \caption{
            Cycle diagrams for permutations $645312$, $53412$, and $456321$,
            with associated compositions $(1,2|1,2)$, $(1,2,2)$, and $(3|1,1,1)$, respectively.
            The line $y=x$ is shown by a dashed line.
        }
        \label{fig:cycle-diagram}
    \end{figure}
\end{ex}

\section{Determining if a Balanced Composition is Cyclic}\label{sec:balanced}

In this section, we will prove specialized versions of our main results
for balanced compositions.
We will generalize these results to all compositions in the next section.

\subsection{Reducing Cycles}

The key result in this subsection is the Cycle Reduction Lemma,
which is Lemma~\ref{lem:cycle-reduction}.

\begin{lem}\label{lem:prep-cycle-reduction}
    Suppose the balanced composition
    \[
        (a_1,\ldots,a_k | b_m,\ldots,b_1)
    \]
    is cyclic, and $a_k = b_m$.  Then $a_k=b_m=k=m=1$ and $n=2$.
\end{lem}
\begin{proof}
    If $a_k=b_m$, then the associated permutation $\pi$
    contains the $2$-cycle $\lt(\f n2, \f n2+b_m\rt)$.
    Since the composition is cyclic, this $2$-cycle is the only cycle,
    and thus $n=2$ and $a_k=b_m=k=m=1$.
\end{proof}

For a balanced composition
\[C = (a_1,\ldots,a_k | b_m,\ldots,b_1)\]
with $a_k \neq b_m$,
we define the \vocab{reduction} operation $\red$ as follows.
Let
\begin{eqnarray*}
    u &=& a_k \mod |a_k - b_m| \\
    v &=& |a_k-b_m| - u.
\end{eqnarray*}
If $a_k>b_m$, define
\[
    \red(C) = (a_1,\ldots,a_{k-1},u,v | b_{m-1},\ldots,b_1).
\]
Analogously, if $a_k<b_m$, define
\[
    \red(C) = (a_1,\ldots,a_{k-1}| v,u,b_{m-1},\ldots,b_1).
\]
In both cases, if $u=0$, omit $u$ from the composition.

Note that $\red$ always decreases a composition's sum,
because $u+v=|a_k-b_m|<a_k+b_m$.

\begin{lem}[Cycle Reduction Lemma]\label{lem:cycle-reduction}
Suppose $a_k\neq b_m$.  Then, the composition
\[
    C = (a_1,\ldots,a_k | b_m,\ldots,b_1)
\]
is cyclic if and only if $\red(C)$ is cyclic.
\end{lem}

Before proving this lemma,
we give an example that captures the spirit of the proof.
\begin{ex}\label{ex:cyc-reduction}
    Let $a_k=5$, $b_m=2$, so $u=2$ and $v=1$.
    The Cycle Reduction Lemma states that a balanced composition
    \[
        C = (a_1,\ldots,a_{k-1},5|2,b_{m-1},\ldots,b_1)
    \]
    is cyclic if and only if
    \[
        \red(C) = (a_1,\ldots,a_{k-1},2,1|b_{m-1},\ldots,b_1)
    \]
    is cyclic.

    Let us focus on the layers corresponding to $a_k=5$ and $b_m=2$,
    which are the innermost layers of the cycle diagram
    on either side of the line $y=x$.

    The wires incident to these two layers connect with each other,
    leaving three pairs of loose ends, as shown in the left diagram
    in Figure~\ref{fig:cycle-reduction-example}.
    The wires induce a natural pairing on these loose ends:
    two loose ends are paired if they are connected by these wires.

    These loose ends are paired the same way by two layers
    of size $u=2$ and $v=1$, as shown in the right diagram
    in Figure~\ref{fig:cycle-reduction-example}.
    \begin{figure}
        \includegraphics{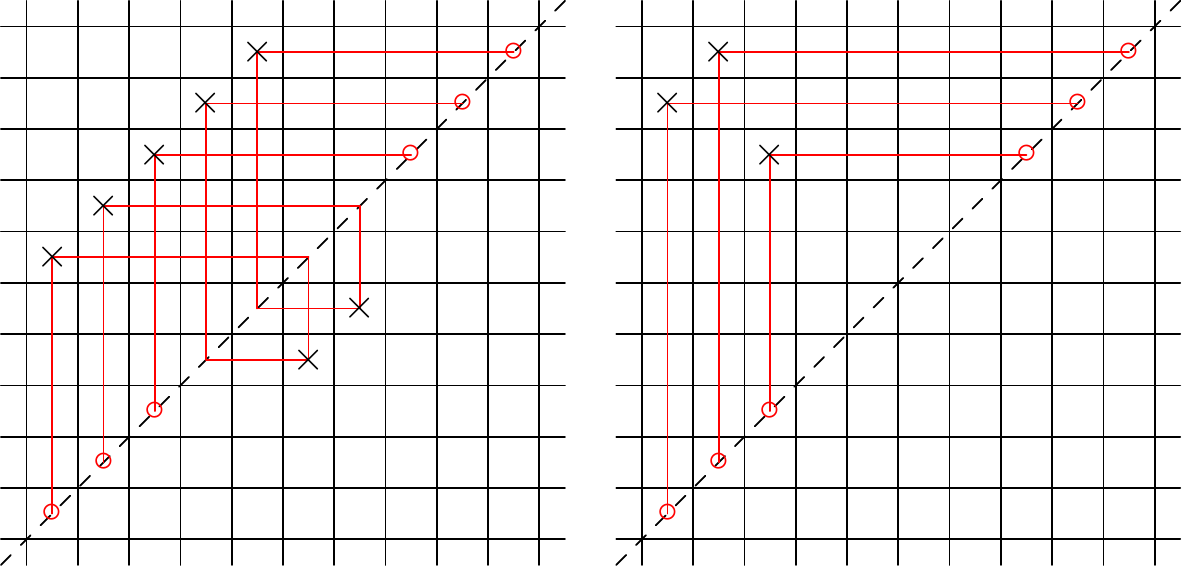}
        \caption{
            Cycle diagrams for a layer of size 5 opposite a layer of size 2 (left),
            and for consecutive layers of size 2 and 1 on the same side of the
            line $y=x$ (right), with loose ends.
            The line $y=x$ is shown by a dashed line, and loose ends are circled.
        }
        \label{fig:cycle-reduction-example}
    \end{figure}
    Consider any configuration of the remaining wires in the cycle diagram,
    which we call the \vocab{outer wiring}.
    These wires also form three pairs of loose ends, which connect with the
    three pairs of loose ends from the innermost wires.
    Thus, once we fix an outer wiring, whether the cycle diagram is cyclic
    depends only on the pairing of loose ends induced by the innermost wires.

    Since the wirings of the innermost layers of $C$ and $\red(C)$
    pair their loose ends the same way,
    the cycle diagram of $C$ is a single cycle if and only if
    the cycle diagram of $\red(C)$ is a single cycle.
\end{ex}

\begin{proof}[Proof of Lemma~\ref{lem:cycle-reduction}]
    Let us assume $a_k > b_m$; the case $a_k < b_m$ is symmetric.
    Let $D = a_k - b_m$.
    Thus $u = a_k \mod D$ and $v = D - u$.
    Let us write $a_k = qD+u$.

    Consider the cycle diagram of $C$.
    The terms $a_k$ and $b_m$ correspond to the two
    innermost layers of the cycle diagram on either side of the line $y=x$.
    Call these layers $L_a$ and $L_b$.
    When we connect the wires incident to these layers,
    each point in $L_b$ is connected to two points in $L_a$,
    which are diagonally $D$ cells apart.

    The rightward wires from the upper-right $a_k-b_m$ points in $L_a$
    and the downward wires from the lower-left $a_k-b_m$ points in $L_a$ are loose ends.
    Call these sets the \vocab{upper-right loose ends}
    and the \vocab{lower-left loose ends}.
    The wiring of $L_a$ and $L_b$ induces a pairing between these sets of loose ends,
    as described in Example~\ref{ex:cyc-reduction}.

    Let us determine this pairing by starting at
    one of the upper-right loose ends and following its wire.
    Every time the wire winds around the center
    of the cycle diagram, it passes through one point in $L_a$
    and one point in $L_b$;
    because each point in $L_b$ is connected to two points in $L_a$
    that are diagonally $D$ cells apart,
    successive points in $L_a$ and on this wire are diagonally $D$ cells apart.

    Let us examine the points this wire passes through in $L_a$.
    If we started at one of the upper-rightmost $u$ loose ends,
    the wire passes through $q+1$ points in $L_a$;
    thus the last point in $L_a$ on the wire is diagonally $qD$ cells
    from the first point in $L_a$ on the wire.
    Moreover, this implies that the upper-rightmost $u$ loose ends
    are paired with the lower-leftmost $u$ loose ends.
    The remaining $D-u=v$ upper-right loose ends pass through $q$ points in $L_a$;
    thus the last point in $L_a$ on their wires is diagonally $(q-1)D$ cells
    from the first point in $L_a$ on their wires.

    This pairing is the same as the pairing produced by two consecutive layers
    of $u$ and $v$ cells, with the layer of size $v$ on the inside.
    Therefore, any fixed outer wiring makes a cyclic wiring when connected
    with opposing layers of size $a_k, b_m$ if and only
    it makes a cyclic wiring when connected with consecutive layers of size $u, v$.
    So, the cycle diagram of $C$ consists of a single cycle
    if and only if the cycle diagram of $\red(C)$ consists of a single cycle.
\end{proof}

\subsection{The Repeated Reduction Algorithm}

Lemmas~\ref{lem:prep-cycle-reduction} and \ref{lem:cycle-reduction}
imply the following algorithm to determine if a balanced composition is cyclic.
This is Algorithm~\ref{alg:main} specialized to balanced compositions.
\begin{alg}[Repeated Reduction Algorithm]\label{alg:cycle-reduction}
    Take as input a balanced composition $C$.
    Repeatedly apply $\red$ to $C$ until $a_k=b_m$.
    If this procedure stops at the composition $(1|1)$, output that $C$ is cyclic.
    Otherwise, output that $C$ is not cyclic.
\end{alg}
We denote this algorithm by $\rred$.
Note that this algorithm must terminate,
because each application of $\red$ decreases the sum of the composition.

We can now prove a necessary condition for a balanced composition to be cyclic.
\begin{prop}\label{prop:main-balanced}
    Every cyclic balanced composition has exactly two odd entries.
\end{prop}
\begin{proof}
    The operation $\red$, and therefore the algorithm $\rred$,
    preserves the number of odd entries in a composition.
    The cyclic balanced compositions all reduce to $(1|1)$,
    so they must themselves have exactly two odd entries.
\end{proof}

\section{The General Setting}\label{sec:unbalanced}

In this section,
we will generalize the results of the previous section to all compositions.

\subsection{Equalization}

The Repeated Reduction Algorithm, developed in the previous section,
determines whether a balanced composition is cyclic.
We now address the problem of determining whether any composition is cyclic.
We do this via \vocab{equalization},
an operation that transforms any composition into a balanced composition.
The Equalization Lemma, which is Lemma~\ref{lem:unbalanced-cycle-reduction},
reduces this new problem to the one addressed in the previous section.

In an unbalanced composition $C=(a_1,\ldots,a_k)$, no prefix sums to $\f n2$.
Therefore there is $i$ such that
\[
    a_1+\cdots+a_{i-1} < \f n2
    \quad
    \text{and}
    \quad
    a_1+\cdots+a_i > \f n2.
\]
We call this $i$ the \vocab{dividing index} of the composition.

For unbalanced compositions, we will develop notions of \vocab{nearly-equal division}
and \vocab{unequalness}.
\begin{defn}
    The \vocab{nearly-equal division} of an unbalanced composition $C$
    with dividing index $i$ is defined as follows.
    \begin{itemize}
        \item If $a_1+\cdots+a_{i-1} \le a_{i+1}+\cdots+a_k$, then
        the nearly equal division is $((a_1,\ldots,a_i),(a_{i+1},\ldots,a_k))$.
        \item If $a_1+\cdots+a_{i-1} > a_{i+1}+\cdots+a_k$, then
        the nearly equal division is $((a_1,\ldots,a_{i-1}),(a_i,\ldots,a_k))$.
    \end{itemize}
\end{defn}
\begin{ex}
    The nearly-equal division of $(1,1,3,1)$ is $((1,1),(3,1))$;
    the $3$ belongs to the second part of the division because $1+1>1$.
    The nearly-equal division of $(3,4,2,2)$ is $((3,4),(2,2))$;
    the $4$ belongs to the first part of the division because $3\le 2+2$.
\end{ex}
\begin{defn}
    The \vocab{unequalness} of an unbalanced composition $C$
    with nearly-equal division $((a_1,\ldots,a_j),(a_{j+1},\ldots,a_k))$ is
    \[
        U(C) = \lt|\sum_{\ell=1}^j a_\ell - \sum_{\ell=j+1}^k a_\ell\rt|.
    \]
\end{defn}
\begin{defn}
    The \vocab{equalization} of an unbalanced composition $C$
    with nearly-equal division $((a_1,\ldots,a_j),(a_{j+1},\ldots,a_k))$ is
    \[
        \eq(C) = (a_1,\ldots,a_j,U(C),a_{j+1},\ldots,a_k),
    \]
    The equalization of a balanced composition $C$ is $\eq(C) = C$.
\end{defn}
Note that $\eq(C)$ is always a balanced composition;
if $C$ is unbalanced, the term $U(C)$ gets added to the smaller side of the
nearly-equal division of $C$, thereby balancing it.
\begin{rem}
    The nearly-equal division is defined non-symmetrically; when
    \[
        a_1+\cdots+a_{i-1} = a_{i+1}+\cdots+a_k,
    \]
    we arbitrarily define the nearly-equal division to be
    $((a_1,\ldots,a_i),(a_{i+1},\ldots,a_k))$.
    However, the definition of equalization is still symmetric.
    When the above equality holds,
    regardless of whether we define the nearly-equal division $C$ as
    \[
        ((a_1,\ldots,a_i),(a_{i+1},\ldots,a_k))
    \]
    or
    \[
        ((a_1,\ldots,a_{i-1}),(a_i,\ldots,a_k)),
    \]
    the equalization is
    \[
        \eq(C)=(a_1,\ldots,a_i,a_i,\ldots,a_k).
    \]
\end{rem}

The following lemma allows us to reduce the question of
whether an unbalanced composition is cyclic
to the question of whether a balanced composition is cyclic.

\begin{lem}[Equalization Lemma]\label{lem:unbalanced-cycle-reduction}
The unbalanced composition $C$ is cyclic if and only if $\eq(C)$ is cyclic.
\end{lem}
Once again, we first demonstrate the lemma with an example.

\begin{ex}
    Suppose $a_1+\cdots+a_{i-1} + 3 = a_{i+1}+\cdots+a_k$, and $a_i=5$.
    Then the nearly-equal division of $C$ is
    \[
        ((a_1,\ldots,a_{i-1},5),(a_{i+1},\ldots,a_k)),
    \]
    and
    \[
        U(C) = \lt| \sum_{\ell=1}^i a_\ell - \sum_{\ell=i+1}^k a_\ell \rt| = 2.
    \]
    Thus,
    \[
        \eq(C) = (a_1,\ldots,a_{i-1},5 | 2,a_{i+1},\ldots,a_k).
    \]
    The left and right diagrams of Figure~\ref{fig:equalization-example} show, respectively,
    the central layer of the cycle diagram of $C$ and the two innermost layers of the cycle
    diagram of $\eq(C)$.
    Note that equalization slides the layer of size $5$ upward and to the left,
    and inserts a layer of size $2$ opposite it,
    while keeping the connectivity of the remaining wires unchanged.

    \begin{figure}
        \includegraphics{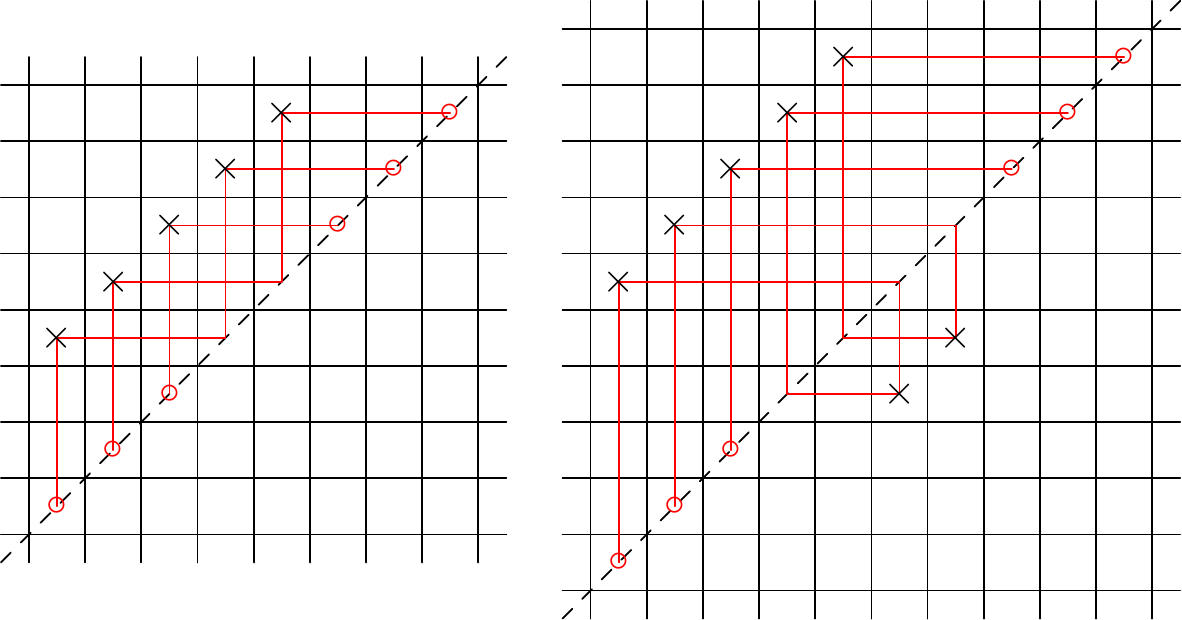}
        \caption{
            Cycle diagram for a central layer of size 5 in a composition with unequalness 2 (left),
            and for a layer of size 5 opposite a layer of size 2 (right).
            The main diagonal is shown by a dotted line, and loose ends are circled.
        }
        \label{fig:equalization-example}
    \end{figure}
    In both diagrams, there are three pairs of loose ends;
    the loose ends are paired by connectivity in the same way.
    Therefore, $C$ is cyclic if and only if $\eq(C)$ is cyclic.
\end{ex}

\begin{proof}[Proof of Lemma~\ref{lem:unbalanced-cycle-reduction}]
    Assume $C$ is unbalanced.
    Let $C=(a_1,\ldots,a_k)$ have dividing index $i$.
    If
    \[
        a_1+\cdots+a_{i-1} = a_{i+1}+\cdots+a_k,
    \]
    then
    \[
        \eq(C) = (a_1,\ldots,a_i,a_i,\ldots,a_k).
    \]
    So, the middle $a_i$ entries of the permutation associated to $C$ are fixed points,
    and the middle $2a_i$ entries of the permutation associated to $\eq(C)$
    form $a_i$ $2$-cycles.
    Therefore if $C=(1)$, $C$ and $\eq(C)$ are simultaneously cyclic,
    and otherwise they are simultaneously not cyclic.

    Otherwise, we further assume
    \[
        a_1+\cdots+a_{i-1} \neq a_{i+1}+\cdots+a_k.
    \]
    The remaining proof is, once again, a matter of following loose ends.
    Let us assume, without loss of generality, that
    \[
        a_1+\cdots+a_{i-1} < a_{i+1}+\cdots+a_k.
    \]
    Then the nearly-equal division of $C$ is
    $((a_1,\ldots,a_i),(a_{i+1},\ldots,a_k))$.
    Let $U=U(C)$.  Note that $U < a_i$, because otherwise
    \[
        \sum_{\ell=1}^i a_\ell < \sum_{\ell=i+1}^k a_\ell,
    \]
    contradicting that $i$ is the dividing index.

    Let us call the layer corresponding to $a_i$ in the cycle diagram of $C$
    the \vocab{central layer} of the cycle diagram; we denote this layer $L_c$.
    Note that points not in $L_c$ are adjacent, along the wire,
    to two points on the opposite side of the line $y=x$.
    This property may only fail to hold for points in $L_c$.

    Since
    \[
        \sum_{\ell=i+1}^k a_\ell - \sum_{\ell=1}^{i-1} a_\ell
        = a_i - U,
    \]
    all points in $L_c$ are $a_i-U$ cells above the main diagonal.
    Therefore, points in $L_c$ that are diagonally $a_i-U$ apart
    are adjacent along the wire.
    Moreover, these connections leave $a_i-U$ pairs of loose ends:
    $a_i-U$ rightward-pointing loose ends
    incident to the upper-rightmost $a_i-U$ points in $L_c$,
    and equally many downward-pointing loose ends
    incident to the lower-leftmost $a_i-U$ points in $L_c$.

    In $\eq(C)$, an additional layer $L_u$ of size $U$ is added
    just below and to the right of $L_c$,
    resulting in a balanced composition where the two innermost layers
    have size $a_i$ and $U$.
    Each point in $L_u$ is adjacent along the wire to two points in $L_c$;
    it is easy to see these points are diagonally $a_i-U$ apart.

    Thus, if in the cycle diagram of $C$,
    two points in $L_c$ are adjacent along the wire,
    in the cycle diagram of $\eq(C)$,
    they are two apart along the wire, separated by a point in $L_u$.
    It follows that the loose ends in $L_c$ are paired the same way
    in the cycle diagrams of $C$ and $\eq(C)$.
    So, the cycle diagram of $C$ is a single cycle if and only if
    the cycle diagram of $\eq(C)$ is a single cycle.
\end{proof}

\subsection{Completion of the Proof}

Lemma~\ref{lem:unbalanced-cycle-reduction}, in conjunction with
Lemma~\ref{lem:cycle-reduction},
immediately implies the validity of Algorithm~\ref{alg:main}.

We will now prove Theorem~\ref{thm:main}, restated below for clarity.
\begin{repthm}{thm:main}
    For all positive integers $n$, $C_n(132,213) \le n^2 \cdot 2^{n/2}$.
\end{repthm}

We first note the following structural result.
\begin{prop}\label{prop:main-unbalanced}
    All cyclic compositions of $n$ have exactly one odd term if $n$ is odd
    and two odd terms if $n$ is even.
\end{prop}
\begin{proof}
    Suppose $C$ is a cyclic composition of $n$.
    If $n$ is odd, the unequalness $U(C)$ is odd;
    if $n$ is even, the unequalness $U(C)$ is even.
    By Proposition~\ref{prop:main-balanced}, $\eq(C)$ has exactly two odd terms.
    Therefore, $C$ has one odd term if $n$ is odd, and two odd terms if $n$ is even.
\end{proof}

\begin{proof}[Proof of Theorem~\ref{thm:main}]
    Suppose $n$ is odd.
    By Proposition~\ref{prop:main-unbalanced},
    cyclic compositions of $n$ have exactly one odd term.
    We can obtain any such composition by decrementing one term
    of a composition of $n+1$ with only even terms.
    There are $2^{(n-1)/2}$ compositions of $n+1$ with only even terms.
    Since each has at most $\f{n+1}{2}$ terms,
    the number of cyclic compositions of $n$ is bounded above by
    \[
        \f{n+1}{2} \cdot 2^{(n-1)/2} \le n^2 \cdot 2^{n/2}.
    \]

    Otherwise, suppose $n$ is even.
    By Proposition~\ref{prop:main-unbalanced},
    cyclic compositions of $n$ have exactly two odd terms.
    We can obtain any such composition by decrementing two terms
    of a composition of $n+2$ with only even terms.
    There are $2^{n/2}$ compositions of $n+2$ with only even terms.
    Since each has at most $\f n2 + 1$ terms,
    we can decrement two terms in at most $\binom{n/2 + 1}{2}$ ways.
    So, the number of cyclic compositions of $n$ is bounded by
    \[
        \binom{n/2 + 1}{2} \cdot 2^{n/2} \le n^2 \cdot 2^{n/2}
    \]
\end{proof}
\begin{rem}
    The proof of Theorem~\ref{thm:main} implies
    an upper bound of $n^2 \cdot 2^{(n-1)/2}$,
    and this bound can be further refined by a constant factor.
    For odd $n$, the proof achieves a tighter bound by a factor of $n$.
    Since, as we discuss in the next section,
    we do not believe our bound's exponential term is tight,
    we are content to drop these factors.
\end{rem}

\section{Future Directions}\label{sec:conclusion}

This paper makes progress towards Problem~\ref{prob:main},
the enumeration of the $(132,213)$-avoiding cyclic permutations.
This problem is still open.

Leveraging Algorithm~\ref{alg:main},
computer experiments done by the author have computed $C_n(132,213)$
for $n$ up to $75$.\footnote{
    Implementation detail:
    this data was computed by a dynamic programming algorithm,
    where the subproblems were to count the number of
    balanced cyclic compositions $(a_1,\ldots,a_k | b_m,\ldots,b_1)$ of $n$
    where the sequence $a_1,\ldots,a_k$ has a specified suffix.
    The runtime of this algorithm is exponential,
    but grows slowly enough that data collection for $n$ up to $75$ is possible.
}
This data is shown in Table~\ref{table:cn-data}.
\begin{table}[H]
    \centering
    \scriptsize
    \begin{tabular}{|c|c||c|c||c|c||c|c||c|c||}
        \hline
        $n$ & $C_n(132,213)$ & $n$ & $C_n(132,213)$ & $n$ & $C_n(132,213)$ & $n$ & $C_n(132,213)$ & $n$ & $C_n(132,213)$ \\ \hline\hline
         1 &   1 & 16 &   762 & 31 &   27892 & 46 &   8501562 & 61 &   129285010 \\ \hline
         2 &   1 & 17 &   440 & 32 &  138200 & 47 &   2570744 & 62 &   803955498 \\ \hline
         3 &   2 & 18 &  1548 & 33 &   49276 & 48 &  15140024 & 63 &   226271426 \\ \hline
         4 &   4 & 19 &   818 & 34 &  252032 & 49 &   4498100 & 64 &  1413400762 \\ \hline
         5 &   6 & 20 &  3060 & 35 &   87276 & 50 &  26777982 & 65 &   395525678 \\ \hline
         6 &  12 & 21 &  1490 & 36 &  459102 & 51 &   7886792 & 66 &  2478240778 \\ \hline
         7 &  14 & 22 &  5960 & 37 &  153586 & 52 &  47470826 & 67 &   692053810 \\ \hline
         8 &  32 & 23 &  2720 & 38 &  827884 & 53 &  13792064 & 68 &  4350163074 \\ \hline
         9 &  30 & 24 & 11404 & 39 &  270876 & 54 &  83680928 & 69 &  1209749736 \\ \hline
        10 &  76 & 25 &  4894 & 40 & 1494032 & 55 &  24162342 & 70 &  7621011834 \\ \hline
        11 &  62 & 26 & 21596 & 41 &  475282 & 56 & 147821872 & 71 &  2116321814 \\ \hline
        12 & 170 & 27 &  8790 & 42 & 2671066 & 57 &  42241704 & 72 & 13362224638 \\ \hline
        13 & 122 & 28 & 40446 & 43 &  835998 & 58 & 259952664 & 73 &  3699626596 \\ \hline
        14 & 370 & 29 & 15654 & 44 & 4784840 & 59 &  73959542 & 74 & 23395287534 \\ \hline
        15 & 232 & 30 & 74906 & 45 & 1464206 & 60 & 457955944 & 75 &  6471271704 \\ \hline
    \end{tabular}
    \caption{Values for $C_n(132,213)$ for $n$ up to $75$.}
    \label{table:cn-data}
\end{table}
Some observations are apparent from this data.
First, the values of $C_n(132,213)$
show different behavior for even and odd $n$.
For $n\ge 8$,
the values for even $n$ are larger than the values for adjacent odd $n$.
This is expected,
because cyclic compositions with odd sum have exactly one odd term,
whereas cyclic compositions with even sum have exactly two odd terms;
the former condition is more restrictive.

Second, the growth of $C_n(132,213)$, for both even and odd $n$,
appears to be asymptotically slower than $2^{n/2}$.
From fitting the data, we get an asymptotic estimate of
\[
    C_n(132,213) = 2^{\alpha n + o(n)},
\]
where $\alpha \approx 0.38$.
We believe this $\alpha$ is the same for even and odd $n$;
this is because the discrepancy $\frac{C_{n}(132,213)}{C_{n-1}(132,213)}$
for even $n$ appears, empirically, to be subexponential (and in fact, sublinear).

Thus, we do not believe the upper bound in Theorem~\ref{thm:main}
is asymptotically tight.
Of course, this prompts the following problem.
\begin{prob}
    What is the correct value of $\alpha$ in the above asymptotic?
\end{prob}
Theorem~\ref{thm:main} implies $\alpha \le 0.5$.
Both an improvement of this bound and a nontrivial lower bound
would be interesting results.

One approach for future research is to study the number-theoretic properties
of Algorithm~\ref{alg:main},
to obtain results in the spirit of Proposition~\ref{prop:main-unbalanced}.
If one can determine more number-theoretic structure of cyclic compositions,
it may be possible to refine the upper bound in Theorem~\ref{thm:main}.

Another more algebraic approach is to obtain
recursive identities or inequalities for values of $C_n(132,213)$.
This approach involves breaking $C_n(132,213)$ into subproblems,
perhaps by suffixes of the sequence $a_1,\ldots,a_k$
in the balanced composition $\eq(C) = (a_1,\ldots,a_k|b_m,\ldots,b_1)$,
and finding injective or bijective mappings among the subproblems.
It may be possible to obtain a lower bound for $C_n(132,213)$ in this manner.

Because the first step of Algorithm~\ref{alg:main} reduces all compositions to a
balanced composition, it would be of independent interest to
enumerate the balanced cyclic compositions of $n$.
These correspond to the balanced reverse layered permutations of length $n$.
Let $C^B_n(132,213)$ be the number of such compositions and permutations.
For even $n$ up to $74$,
computer experiments give the values of $C^B_n(132,213)$
in Table~\ref{table:cbn-data}.
\begin{table}[H]
    \centering
    \scriptsize
    \begin{tabular}{|c|c||c|c||c|c||c|c||c|c||}
        \hline
        $n$ & $C^B_n(132,213)$ & $n$ & $C^B_n(132,213)$ & $n$ & $C^B_n(132,213)$ & $n$ & $C^B_n(132,213)$ & $n$ & $C^B_n(132,213)$ \\ \hline\hline
         2 &   1 & 18 &   586 & 34 &   86572 & 50 &   8948694 & 66 &  821844316 \\ \hline
         4 &   2 & 20 &  1140 & 36 &  158146 & 52 &  15884762 & 68 & 1442300988 \\ \hline
         6 &   6 & 22 &  2182 & 38 &  281410 & 54 &  27882762 & 70 & 2525295380 \\ \hline
         8 &  14 & 24 &  4130 & 40 &  509442 & 56 &  49291952 & 72 & 4426185044 \\ \hline
        10 &  34 & 26 &  7678 & 42 &  901014 & 58 &  86435358 & 74 & 7747801190 \\ \hline
        12 &  68 & 28 & 14368 & 44 & 1618544 & 60 & 152316976 &    & \\ \hline
        14 & 150 & 30 & 26068 & 46 & 2852464 & 62 & 266907560 &    & \\ \hline
        16 & 296 & 32 & 48248 & 48 & 5089580 & 64 & 469232204 &    & \\ \hline
    \end{tabular}
    \caption{Values for $C^B_n(132,213)$ for even $n$ up to $74$.}
    \label{table:cbn-data}
\end{table}

The data suggests the following conjecture.
\begin{conj}
    For even $n$,
    \[\f{C^B_n(132,213)}{C_n(132,213)} = \Omega(1).\]
\end{conj}
That is, the proportion of $(132,213)$-avoiding cyclic permutations
that are balanced is bounded below by a constant,
which is empirically about $0.33$.

\subsection*{Acknowledgements}

This research was completed in
the 2018 Duluth Research Experience for Undergraduates (REU) program,
and was funded by NSF/DMS grant 1650947 and NSA grant H98230-18-1-0010.
The author gratefully acknowledges Joe Gallian suggesting the problem
and supervising the research.
The author thanks Colin Defant and Levent Alpoge for useful discussions
over the course of this work, and
Joe Gallian and Danielle Wang for comments on early drafts of this paper.
The author thanks the anonymous reviewers for useful comments and suggestions.

\end{document}